\documentclass[12pt]{amsart}
\usepackage[active]{srcltx}
\usepackage{verbatim}
\usepackage{epsfig,graphicx,color,mathrsfs}
\usepackage{graphicx}
\usepackage{amsmath,amssymb,amsthm,amsfonts}
\usepackage{amssymb}
\usepackage[english]{babel}
\usepackage{epsfig,graphicx,color,mathrsfs}
\usepackage[english]{babel}
\usepackage[left=2.7cm,right=2.7cm,top=3cm,bottom=3cm]{geometry}

\newcommand{\R}{{\mathbb R}}

\newcommand{\teta }{\theta }

\numberwithin{equation}{section}
\newtheorem{theorem}{Theorem}[section]
\newtheorem{proposition}[theorem]{Proposition}
\newtheorem{lemma}[theorem]{Lemma}
\newtheorem{corollary}[theorem]{Corollary}

\newtheorem{remark}[theorem]{Remark}

\theoremstyle{definition}

\newcommand{\brm}{\begin{remark}\rm}
\newcommand{\erm}{\end{remark}}
\newcommand{\brms}{\begin{remark}\rm}
\newcommand{\erms}{\end{remark}}
\newcommand{\bte}{\begin{theorem}}
\newcommand{\ete}{\end{theorem}}
\newcommand{\bpr}{\begin{proposition}}
\newcommand{\epr}{\end{proposition}}
\newcommand{\ble}{\begin{lemma}}
\newcommand{\ele}{\end{lemma}}
\newcommand{\beq}{\begin{equation}}
\newcommand{\eeq}{\end{equation}}
\newcommand{\bdm}{\begin{displaymath}}
\newcommand{\edm}{\end{displaymath}}
\numberwithin{equation}{section}

\newcommand{\bos}{\begin{remark}\rm}
\newcommand{\eos}{\end{remark}}

\newcommand{\ben}{\begin{enumerate}}
\newcommand{\een}{\end{enumerate}}

\newcommand{\be}{\begin{equation}}
\newcommand{\ee}{\end{equation}}

\title[Monotonicity in half-spaces]{Monotonicity in half-spaces of positive solutions to $-\Delta_p u=f(u)$ in the case $p>2$}

\author[A.\ Farina]{Alberto Farina$^+$}
\address{Universit\'e de Picardie Jules Verne
\newline\indent
LAMFA, CNRS UMR 6140\newline\indent
Amiens, France}
\email{alberto.farina@u-picardie.fr}

\author[L.\ Montoro]{Luigi Montoro$^*$}
\address{Dipartimento di Matematica
\newline\indent
Universit\`a della Calabria
\newline\indent
Ponte Pietro Bucci 31B, I-87036 Arcavacata di Rende, Cosenza, Italy}
\email{montoro@mat.unical.it}

\author[B.\ Sciunzi]{Berardino Sciunzi$^*$}
\address{Dipartimento di Matematica
\newline\indent
Universit\`a della Calabria
\newline\indent
Ponte Pietro Bucci 31B, I-87036 Arcavacata di Rende, Cosenza, Italy}
\email{sciunzi@mat.unical.it}


\begin{document}

\begin{abstract}
We consider  weak distributional solutions to the equation
$-\Delta_pu=f(u)$  in half-spaces under  zero Dirichlet boundary condition.  We assume that the nonlinearity is positive and superlinear at zero.
For $p>2$ (the case $1<p\leq2$ is already known) we prove that any positive solution is strictly monotone increasing in the direction orthogonal to the boundary of the half-space.
As a consequence we deduce some Liouville type theorems for the Lane-Emden type equation. Furthermore any nonnegative solution turns out to be $C^{2,\alpha}$ smooth.
\end{abstract}

\maketitle

\section{Introduction}\label{introdue}
We consider the problem
\begin{equation}\label{E:P}
\begin{cases}
-\Delta_p u=f(u), & \text{ in }\mathbb{R}^N_+\\
u(x',y) \geqslant 0, & \text{ in } \mathbb{R}^N_+\\
u(x',0)=0,&  \text{ on }\partial\mathbb{R}^N_+
\end{cases}
\end{equation}
where $N \geq 2$ and  $f(\cdot)$ satisfies:

\

\begin{itemize}
\item[($h_f$)] the nonlinearity $f$ is \emph{positive} i.e.  $f(t)>0$ for $t>0$, locally Lipschitz continuous in $\mathbb{R}^+\cup \{0\}$ and
\[
 \lim_{t\rightarrow 0^+}\frac{f(t)}{t^{p-1}}=f_0\in\mathbb{R}^+\cup\{0\}.
 \]
\end{itemize}

\

\noindent In the following we denote a generic point in
$\mathbb{R}^N$ by $(x',y)$ with $x'=(x_1,x_2, \ldots, x_{N-1})$ and $y=x_N$, we assume with no loss of generality that
$\mathbb{R}^N_+=\{y>0\}$.
Furthermore, according to the regularity results in \cite{Di,Li,T} (see also the recent developments in \cite{kus,tex}), we assume that $u \in C^{1,\alpha}_{loc}(\overline{\mathbb{R}^N_+})$ and fulfills the equation in the weak distributional meaning. Actually in our case the regularity up to the boundary does not follow directly by \cite{Li} and an argument by reflection is needed. This is quite standard and will be described also later on in this paper. \\
\noindent By the strong maximum principle \cite{V}, it follows that any nonnegative nontrivial solution is actually (strictly) positive. In this case: \emph{we study the monotonicity of the solution in the direction orthogonal to the boundary of the half-space}.\\

\noindent  The main tool is
  the Alexandrov-Serrin moving plane method that goes back to \cite{A,S}.
  It is well known that the moving plane procedure allows to prove monotonicity and symmetry properties of the solutions to general PDE. In the case of  bounded domains and in the semilinear case $p=2$, this study was started in the celebrated papers  \cite{BN,GNN}.
  In the case of unbounded domains
  the main examples, arising from many applications,  are provided by the whole space $\mathbb{R}^N$ and by the half-space $\mathbb{R}^N_+$. For the case of the whole space, where radial symmetry of the solutions is expected, we refer to \cite{CGS,GNN,GNN2}.
In this paper we will  address the case when the domain is an half-space. We refer the readers
 to \cite{BCN1,BCN2,BCN3,DaGl,Dan1,Dan2,Fa,FV2,QS}  for previous results concerning  monotonicity of the solutions
in half-spaces, in the non-degenerate case.\\

\noindent The case of $p$-Laplace equations is really  harder to study.  In fact the $p$-laplacian is a nonlinear operator and, as a consequence, comparison principles are not equivalent to maximum principles. The degenerate nature of the operator also causes the lack of regularity of the solutions.  Furthermore, in the case $p>2$ that we are considering, the use of weighted Sobolev spaces is naturally associated to the study of qualitative properties of the solutions.
This issue is more delicate in unbounded domains. We cannot describe with more details this fact that will be clarified  while reading the paper. Let us only say that, the use of weighted Sobolev spaces is necessary in the case $p>2$ and it requires the use of a weighted Poincar\'{e} type inequality with weight $\rho = |\nabla u|^{p-2}$ (see \cite{DS1}). The latter involves constants that may blow up when the solution approaches zero that may happen also for positive solutions in unbounded domains. Namely once again the lack of compactness plays an important role.\\

\noindent First results in bounded domains and in the case $1<p<2$ were obtained in \cite{DP}. The case $p>2$ requires the above mentioned use of weighted Sobolev spaces and was solved in \cite{DS1}, for positive nonlinearities ($f(t)>0$ for $t>0$).
In the case of the whole space, we refer the readers to the recent results in \cite{AdvS,Sci3,VETOIS}.\\

\noindent Considering the $p$-Laplace operator and problems in half-spaces, first results have been obtained in \cite{DS3} in dimension two. The same technique has been also exploited in the fully nonlinear case in \cite{charro}. In higher dimensions, first results have been obtained in the singular case $1<p<2$ in \cite{FMS,FMRS} (see also \cite{galakhov}) where
positive locally Lipschitz continuous nonlinearities are considered. A partial answer in the more difficult degenerate case $p>2$ was obtained in \cite{FMS3}, where \emph{power-like nonlinearities} are considered  under the restriction $2<p<3$. Here, considering a larger class of nonlinearities,
namely considering positive nonlinearities that are superlinear at zero, we remove the condition $2<p<3$ and prove the following:
\begin{theorem}\label{thm:monotoniadellau}
Let $p>2$ and let $u\in C^{1,\alpha}_{loc}({\overline {{\mathbb{R}^N_+}}})$ be
a positive solution to \eqref{E:P} with $|\nabla u|\in L^{\infty}(\mathbb{R}^N_+)$.
 Then, under the assumption $(h_f)$, it follows that
 \[
\frac{\partial u}{\partial y}\,>\,0 \quad \text{in}\quad \mathbb{R}^N_+.
 \]
 As a consequence $u \in C^{2,\alpha'}_{loc}(\overline{\mathbb{R}^N_+})$ for some $0<\alpha '<1$.
\end{theorem}
Our monotonicity result holds in particular for Lane-Emden type equations, namely in the case $f(u)=u^q$ with $q\geq p-1$. Note that, the case
$q\leq p-1$, or more generally the case when, for some $t_0>0$, it holds
\begin{equation}\nonumber
f(t)\geq c\,t^{p-1}\qquad\text{for}\,\,\,t\in[0,t_0]\,,
\end{equation}
is already contained  in \cite[Theorem 3]{FMS3}. Furthermore Theorem  \ref{thm:monotoniadellau} is proved without a-priori assumptions on the behavior of the solution, that is,   at infinity the solution may decay at zero in some regions, while it can be far from zero in some other regions. Furthermore it is crucial the fact that only local regularity on the solution is required in our result.
 Note in fact that,  assuming that the solution has summability properties at infinity, namely  assuming that the solution belongs to some Sobolev space, then the monotonicity result is somehow more easy to deduce and it
 generally leads to the nonexistence of such solutions, we refer to \cite{mercuri} (see also \cite{Zou}). Finally it is worth emphasizing that we prove the first step of the moving plane procedure in a very general setting.
 In fact, in Theorem \ref{lem:saliacavallo}, we prove that any positive solution is monotone increasing near the boundary for any $1<p<\infty$ and assuming only that
the nonlinearity $f$ is merely  continuous in $\mathbb{R}^+\cup \{0\}$ such that, for some $T>0$, it holds that
$|f(t)|\leq \bar k\, t^{p-1}$ for $t\in[0,T]$
 and for some $\bar k=\bar k(T)>0$.\\

\noindent The technique developed to prove Theorem \ref{thm:monotoniadellau} also allows to deduce a monotonicity result for solutions to equations involving a different class of nonlinearities. We have the following
\begin{theorem}\label{mainthmnonpositive}
Let  $p>2$ and let $u \in C^{1,\alpha}_{loc}(\overline{\mathbb{R}^N_+})\cap W^{1,\infty}(\mathbb{R}^N_+)$ be a positive solution to \eqref{E:P}.
Suppose that $f(\cdot)$ is  locally Lipschitz continuous in $\mathbb{R}^+\cup \{0\}$ and there exists $  t_0 > 0 $ such that
\[
 f(s)> 0 \qquad \text{for }\quad  0<t<t_0\qquad\text{and}\qquad f(s)< 0 \qquad \text{for }\quad  t>t_0\,.
\]
Assume furthermore that
\begin{equation}\label{dnvkjdvkvkvks}
 \lim_{t\rightarrow 0^+}\frac{f(t)}{t^{p-1}}=f_0\in\mathbb{R}^+\cup\{0\} \qquad \text{and}\qquad\lim_{t\rightarrow t_0}\frac{f(t)}{(t_0-t)|t_0-t|^{p-2}}=f^0\in\mathbb{R}^+\cup\{0\}\,.
 \end{equation}
 Then
 \[
\frac{\partial u}{\partial y}\,>\,0 \quad \text{in}\quad \mathbb{R}^N_+.
 \]
  As a consequence $u \in C^{2,\alpha'}_{loc}(\overline{\mathbb{R}^N_+})$ for some $0<\alpha '<1$.
\end{theorem}

Theorem \ref{mainthmnonpositive} is mainly a corollary of Theorem \ref{thm:monotoniadellau} and it extends to the degenerate case $p>2$ earlier results in  \cite{FMRS} (see Theorem 1.3 there and see also Theorem 1.8 in \cite{FMS}). It applies, for instance, to solutions of
\begin{equation}\nonumber
-\Delta_p u=  u(1-u^2)\vert 1 - u^2 \vert^q,
\end{equation}
where $ q \geq p-2$. When $p=2$ and $q = 0$, the above equation reduces to
$$-\Delta u=u(1-u^2)$$
which is the celebrated Allen-Cahn equation arising in a famous conjecture of De Giorgi.

\


\noindent The monotonicity of the solution  implies in particular  the stability of the solution, see  \cite{DFSV,FSV}. This allows us to deduce some Liouville type theorems. Following \cite{DFSV,Fa2}, we  set
\begin{equation}\nonumber
q_c(N,p)=\frac{[(p-1)N-p]^2+p^2(p-2) -p^2(p-1)N+2p^2\sqrt{(p-1)(N-1)}}{(N-p)[(p-1)N-p(p+3)]}\,.
\end{equation}
We refer to  \cite{DFSV,Fa2} and the references therein for more details and we only note here that
the exponent $q_c(N,p)$ is larger than the classical  critical Sobolev exponent.
Once that, by Theorem \ref{thm:monotoniadellau}, we know  that the solutions are monotone (and therefore stable), then the same proof of \cite[Theorem 4]{FMS3} provides the following Liouville-type result:
\begin{theorem}\label{liouvillenextgenerationtris}
Let  $p>2$ and let $u\in C^{1,\alpha}_{loc}({\overline {{\mathbb{R}^N_+}}})$ be a non-negative
weak solution of~\eqref{E:P} in~$\R^N_+$ with $|\nabla u|\in L^\infty (\R^N_+)$ and
$$f(u)=u^q.$$
Assume that
$$
\begin{cases}
&(p-1) < q <\infty,\quad\quad\quad \,\,\,\,\text{if }\qquad \displaystyle N \leqslant  \frac{p(p+3)}{p-1},\\
&(p-1) < q < q_c(N,p), \quad \text{if }\qquad \displaystyle  N >\frac{p(p+3)}{p-1}\,,\end{cases}$$

\noindent then $u=0$.\\

\noindent If moreover we assume that $u$ is bounded, then it follows that $u=0$ assuming only that
$$
\begin{cases}
&(p-1) < q <\infty,\quad\quad\quad \quad\qquad\,\,\,\,\text{if }\qquad \displaystyle(N-1) \leqslant  \frac{p(p+3)}{p-1},\\
&(p-1) < q < q_c((N-1),p), \quad \,\,\text{if }\qquad \displaystyle (N-1) >\frac{p(p+3)}{p-1}\,.\end{cases}$$
\end{theorem}

\

\noindent The paper is organized as follows. In Section \ref{secpre} we recall some known results for the reader's covenience. In Section \ref{secmain} we prove some preliminary results and then we prove Theorem~\ref{thm:monotoniadellau} and Theorem \ref{mainthmnonpositive}.

\section{Preliminaries}\label{secpre}
We start stating    some notations and   preliminary results. Generic fixed and numerical constants will be denoted by
$C$ (with subscript in some case) and they will be allowed to vary within a single line or formula.
\\

\noindent For $0\leq \alpha<\beta$, define the strip $\Sigma_{(\alpha,\beta)}$ as
\begin{equation}\label{eq:strippppppp}
\Sigma_{(\alpha,\beta)}:= \mathbb{R}^{N-1}\times (\alpha,\beta)
\end{equation}
and   we will indicate with $\Sigma_\beta$    the strip
\begin{equation}\nonumber
\Sigma_{\beta}:= \mathbb{R}^{N-1}\times (0,\beta).
\end{equation}
Then   we define the cylinder
\begin{equation}\label{eq:cylllllllinder}
\mathcal{C}_{(\alpha, \beta)}(R)=\mathcal{C}(R):=\Sigma_{(\alpha, \beta)}\cap \{B^{'}(0,R)\times \mathbb{R}\}\,,
\end{equation}
where $B^{'}(0,R)$ is the ball   in $\mathbb{R}^{N-1}$ of radius $R$ and center at zero.
Given $\lambda \in \mathbb{R}$ we will define   $u_\lambda(x)$  by  \begin{equation}\label{eq:reflected}
u_\lambda(x)=u_\lambda(x',y)\,:=\, u(x',2\lambda-y)\,\quad\text{in}\,\,\, \Sigma_{2\lambda}\,.
\end{equation}
Finally we use the notation
$$u^+:=\max\{u,0\}.$$
In the sequel of the paper   we will often use the strong maximum principle.  We refer to    \cite{V}   (see also  \cite{PSB}) and we recall here the statement.
\begin{theorem}(Strong Maximum Principle and Hopf's Lemma).\label{semihop}
 Let $\Omega$ be a domain in $\mathbb{R}^N$ and suppose that $u \in C^1(\Omega)$, $u \geqslant 0$
 in $\Omega$, weakly solves
 \[
 -\Delta_p u+cu^q=g \geqslant 0 \quad \mbox{in  }\quad \Omega\,,
 \]
 with $1 < p < \infty$, $q \geqslant p-1$, $c \geqslant 0$ and $g \in L^\infty_{loc}(\Omega)$. If
 $u \neq 0$ then $u >0$ in $\Omega$. Moreover for any point $x_0 \in \partial \Omega$ where the
 interior sphere condition is satisfied, and such that $u \in C^1(\Omega \cup \{x_0\})$ and
 $u(x_0)=0$ we have that $\frac{\partial u}{\partial s}>0$ for any inward directional derivative
 (this means that if $y$ approaches $x_0$ in a ball $B \subseteq \Omega$ that has $x_0$ on its
 boundary, then $\displaystyle \lim_{y \rightarrow x_0}\frac{u(y)-u(x_0)}{|y-x_0|}>0$).
\end{theorem}
Let us recall that the linearized operator  $L_u(v,\varphi)$ at a fixed solution $u$ of $-\Delta_p (u)=f(u)$ is well
defined, for every  $v\, ,\, \varphi\in H^{1,2}_{\rho}(\Omega)$  with $\rho\equiv |\nabla u|^{p-2}$, by \ \ \
$$
L_u(v,\varphi) \equiv\int_{\Omega}[|\nabla u|^{p-2}(\nabla v,\nabla \varphi)+(p-2)|\nabla u|^{p-4}(\nabla u,\nabla v)(\nabla u,\nabla \varphi) - f'(u)v\varphi]\,.
$$
We refer \cite{DS1} for more details and in particular for the definition of the weighted Sobolev spaces involved. Let us only recall here that the space $H^{1,2}_{\rho}(\Omega)$ can be defined as the space of functions $v$ such that $\|v\|_{H^{1,2}_{\rho}(\Omega)}$ is bounded and
\[
\|v\|_{H^{1,2}_{\rho}(\Omega)}\,:=\, \|v\|_{L^2(\Omega)}\,+\|\nabla v\|_{L^2(\Omega,\rho)}\,.
\]
This is the same space obtained performing the completion of smooth functions under the norm above. The space
$H^{1,2}_{0,\rho}(\Omega)$ is obtained taking the closure of $C^\infty_c(\Omega)$ under the same norm and $\|\nabla v\|_{L^2(\Omega,\rho)}$ is an equivalent norm in $H^{1,2}_{0,\rho}(\Omega)$.\\

\noindent Moreover, $v\in H^{1,2}_{\rho}(\Omega)$  is a weak solution of the linearized equation if
\begin{equation}\nonumber
L_u(v,\varphi)=0
\end{equation}
for any $\varphi\in H^{1,2}_{0,\rho}(\Omega)$.  By \cite{DS1} we have that $u_{x_i}\in H^{1,2}_{\rho}(\Omega)$
 for $i=1,\ldots ,N$, and
$L_u(u_{x_i},\varphi)$ is well defined
 for every $\varphi\in  H^{1,2}_{0,\rho}(\Omega)$, with
\begin{equation}\nonumber
L_u(u_{x_i},\varphi) =0\quad\quad\forall \varphi\in H^{1,2}_{0,\rho}(\Omega).
\end{equation}
In other words, the derivatives of $u$ are weak solutions of the linearized equation. Consequently by the strong maximum principle for the linearized operator (see \cite{DS2}) we have the following
\begin{theorem}\label{hthPMFderrr}
Let  $u\in C^1(\overline{\Omega})$  be a weak solution of $-\Delta_p (u)=f(u)$ in a bounded smooth domain $\Omega$ of
$\mathbb{R}^N$ with $\frac{2N+2}{N+2}<p<\infty$, and $f$ positive ($f(s)>0$ for $s>0$) and locally Lipschitz continuous. Then, for any $i \in \{1,
\dots ,N \}$ and any domain $\Omega '\subset\Omega$ with $ u_{x_{i}}\geqslant 0$ in $\Omega '$, we have that
either $u_{x_{i}} \equiv 0$ in $\Omega '$ or $u_{x_{i}} >0$ in $\Omega '$.
\end{theorem}

We state now the Weighted Poincar\'e type  inequality  proved in \cite{DS1} that will be useful in  the sequel.

\begin{theorem}[Weighted Poincar\'e type inequality]\label{thm:Poincar}
Let $w\in H^{1,2}_{\rho}(\Omega)$ be such that
\begin{equation}\label{eq:hip_real}|w(x)|\leq \hat C \int_{\Omega}\frac{|\nabla w(y)|}{|x-y|^{N-1}}dy,
\end{equation}with  $\Omega$ a bounded domain and $\hat C$ a positive constant. Let $\rho$ be  a  weight function such that
\begin{eqnarray}\label{eq:weight}
\int_{\Omega} \frac{1}{\rho^\tau|x-y|^{\gamma}}\,dy\leq C^*,\qquad\text{for any}\quad x\in\Omega
\end{eqnarray}
with $\max\{(p-2)\, ,\,0\}\leqslant \tau<p-1$, $\gamma< N-2$ $(\gamma =0 \,\,\text{if}\,\, N=2)$.
Then
\begin{equation}\label{eq:poincareineq}
\int_{\Omega}w^2\leq {C_p}\int_{\Omega}\rho|\nabla w|^2,
\end{equation}
where $C_p={C_p}(d,C^*)$, with $d=\text{diam}\, (\Omega)$. Futhermore
\begin{equation}\nonumber
C_p\rightarrow 0\quad  \text{if} \,\,  d \rightarrow 0.
\end{equation}
\end{theorem}
We remark  that, for the sake of simplicity and for the reader's convenience, here  we write  explicitly   the dependence of  $C_p$ on the parameters that in the sequel will play a crucial role and that we need to control. The other parameters involved are fixed in our application and we refer the readers to Theorem 8 and to Corollary 2 in Section 5 of  \cite{FMS3} (see also \cite{DS1}).\\

\noindent We will use the weighted Poincar\'e type inequality with $\rho=|\nabla u|^{p-2}$. Next proposition  gives some sufficient conditions in order to satisfy \eqref{eq:weight}.
\begin{proposition}\label{pro:SobConstant}
Let $1<p<\infty$ and  $u\in C^{1,\alpha}(\Omega)$ a  weak  solution to
\begin{equation}\nonumber
-\Delta_p u=h(x) \quad \text{in}\,\, \Omega,
\end{equation}
with $h\in W^{1,\infty} (\Omega)$.
 Let $\Omega'\subset \subset \Omega$ and $0<\delta <dist(\Omega',\partial\Omega)$ and assume that $h>0$ in $\overline{\Omega'_\delta}$, where
$$\Omega'_\delta=\{x\in \Omega : d(x,\Omega')< \delta\}\subset \subset \Omega.$$
Let us fix $\beta_1, \beta_2$ such that
\begin{equation}\nonumber
\underset{x\in \Omega'_\delta}{\inf}\, h(x)\geq \beta_1>0 \quad \text{and}\quad \delta\geq\beta_2>0.
\end{equation}
 Then there exits a positive constant $C^*=C^*(\beta_1,\beta_2)$  such that
\begin{equation}\nonumber 
\begin{split}
&\int_{\Omega'} \frac{1}{|\nabla u|^{\tau}}\frac{1}{|x-y|^\gamma}\leqslant C^*,
\end{split}
\end{equation}
with $\max\{(p-2)\, ,\,0\}\leqslant \tau<p-1$.
\end{proposition}
\begin{remark}
The proof of Proposition \ref{pro:SobConstant} follows by \cite{DS1} (see also \cite{Sci1,Sci2}). Actually for the version that we stated here  we refer to Proposition 1 in Section 4 of \cite{FMS3}. Let us also point out that, as above, we prefer to omit the dependence of the constant $C^*$ on other parameters that are fixed and therefore not relevant in our application.
\end{remark}

Later we will frequently exploit the classical Harnack inequality for $p$-Laplace equations. We refer to \cite{PSB}[Theorem 7.2.1] and the references therein. At some point, as it will be clear later, it will be crucial the use of a boundary type Harnack inequality. Therefore
we state here an adapted version of the  more general and deep result of M.F. Bidaut-V\'eron, R. Borghol and L. V\'eron, see Theorem 2.8 in \cite{bidaut}.
\begin{theorem}[Boundary Harnack Inequality]\label{thm:BoundaryHarnack}
Let $R_0>0$ and  define the cylinder $\mathcal C_{(0,L)}(2R_0)$ and let $u$ be such that
\begin{equation}\nonumber
-\Delta_pu=c(x)u^{p-1}\qquad\text{in}\quad \mathcal C_{(0,L)}(2R_0),
\end{equation}
 with $u$ vanishing on  $\mathcal C_{(0,L)}(2R_0)\cap\{y=0\}$ and with  $\|c(x)\|_{L^{\infty}(C_{(0,L)}(2R_0))}\leq C_0$. Then
\begin{equation}\nonumber
\frac 1C \frac{u(z_2)}{\rho (z_2)}\leq \frac{u(z_1)}{\rho (z_1)}\leq C \frac{u(z_2)}{\rho (z_2)},\quad \forall \, z_1,z_2\in B_{R_0}\cap \mathcal C_{(0,L)}(2R_0) \, : \, 0<\frac{|z_2|}{2}\leq |z_1|\leq 2|z_2|\,,
\end{equation}
where $C=C(p,N,C_0)$ and $\rho(\cdot)$ is the distance function to $\partial\mathbb{R}^N_+$ .
\end{theorem}

\noindent Finally, we state a lemma that will be useful in the proof of Proposition \ref{lem:khjadsjhdfhjdf} below, see \cite[Lemma 2.1]{FMS}.
\begin{lemma}\label{Le:L(R)}
Let $\theta >0$ and $\nu>0$ such that $\theta < 2^{-\nu}$.
Let
$$\mathcal{L}:(1, + \infty) \rightarrow \mathbb{R}$$
be a non-negative and non-decreasing function such that
\begin{equation}\nonumber
\begin{cases}
\mathcal{L}(R)\leq \theta \mathcal{L}(2R) & \forall R>1,\\
\mathcal{L}(R)\leq CR^{\nu} & \forall R >1\,.
\end{cases}
\end{equation}
 Then
$$\mathcal{L}(R)=0.$$

\end{lemma}

\section{Proof of Theorem \ref{thm:monotoniadellau}}\label{secmain}
At the end of  this Section we will give the proof of Theorem \ref{thm:monotoniadellau}. Now we start showing that any  positive
solution to \eqref{E:P} is increasing in the $y-$direction near the boundary $\partial \mathbb{R}^N_+$. We prove such a result
for problems involving a
 more general class of nonlinearities and for any $1<p<\infty$.
We have the following
\begin{theorem}\label{lem:saliacavallo}
Let $1<p<\infty$ and let $u\in C^{1,\alpha}_{loc} (\overline{\mathbb{R}^N_+})$ be a positive weak  solution to \eqref{E:P} with $|\nabla u|\in L^{\infty}(\mathbb{R}^N_+)$. Assume that
the nonlinearity $f$ is   continuous in $\mathbb{R}^+\cup \{0\}$ and, for some $T>0$, it holds that
\[
|f(t)|\leq \bar k\, t^{p-1}\qquad \text{for}\quad t\in[0,T]
 \]
for some $\bar k=\bar k(T)>0$.
Then it follows that there  exits $\lambda >0$ such that
\begin{equation}\label{eq:garyyy}
\frac{\partial u}{\partial y}(x',y)>0\quad \text{in}\,\,\, \Sigma_\lambda.
\end{equation}
In particular the result holds true under the condition $(h_f)$.
\end{theorem}
\begin{proof}
We argue
by contradiction and  we assume that there exists a sequence of points $P_n=(x_n',y_n)$ such that
\begin{equation}\label{eq:salisulcavallo}
\frac{\partial u}{\partial y}(x_n',y_n)\leq0\quad \text{and}\quad y_n\underset{n\rightarrow +\infty}{\longrightarrow} 0.
\end{equation}
We consider the sequence $\hat x_n$ defined by $$\hat x_n=(x_n',1)\,.$$
\noindent We   set $$\alpha_n:=u(x'_n, 1)$$ and
\begin{equation}\label{eq:ashgghagha}
w_n(x',y)=\frac{u(x'+x_n', y)}{\alpha_n}.
\end{equation}
We remark that $w_n(0,1)=1$ and we have
\begin{eqnarray}\label{sdjkfgsfgk}
-\Delta_pw_n(x)&=&\frac{1}{\alpha_n^{p-1}}f\big(u(x'+x_n', y)\big)=\frac{1}{\alpha_n^{p-1}}\frac{f\big(u(x'+x_n', y)\big)}{u^{p-1}(x'+x_n', y)}u^{p-1}(x'+x_n', y)\\\nonumber
&=&c_n(x)w_n^{p-1}(x),
\end{eqnarray}
for
\begin{equation}\label{eq:cccccc(x)}
c_n(x)=\frac{f\big(u(x'+x_n', y)\big)}{u^{p-1}(x'+x_n', y)}\,.
\end{equation}
Since for any $L>0$ we have that $u\in L^{\infty}(\Sigma_{(L)})$ (by the Dirichlet condition and because $|\nabla u|$ is bounded in $\mathbb{R}^N_+$),  by the assumption on the nonlinearity $f$, we obtain that
\begin{equation}\label{eq:lacarrozzellaxxx}
\|c_n(x)\| _{L^{\infty}(\Sigma_{L})}\leq
\|c_n(x)\| _{L^{\infty}(\Sigma_{2L})}\leq C_0(L).
\end{equation}
Now we consider real numbers $L,R$ and $R_0$ satisfying
\begin{equation}\label{LReR0}
0 < 2R_0 < 1 < R < L
\end{equation}

\noindent \emph{We claim that:}
\begin{equation}\nonumber
\|w_n\|_{L^{\infty}(\mathcal C_{(0,L)}(R))}\leq C(L,R,R_0)\,.
\end{equation}
 Since $w_n(0,1)=1$, by the classical Harnack inequality, see \cite{PSB}[Theorem 7.2.1], we have that
\begin{equation}\label{eq:classicHarnack}
\|w_n\|_{L^{\infty}(\mathcal C_{(0,L)}(R)\cap \{y\geq \frac{ R_0}{4}\})}\leq C^{i}_H(L,R, R_0)\,.
\end{equation}
Now we exploit Theorem \ref{thm:BoundaryHarnack} to deduce that
 \begin{equation}\label{eq:boundaryHarnack}
\|w_n\|_{L^{\infty}(\mathcal C_{(0,L)}(R)\cap \{y\leq \frac{ R_0}{4}\})}\leq C^{b}_H(L,R,R_0).
\end{equation}
To this end, let $\tilde P=(\tilde x',\tilde y)$ be such that $\tilde x'\in B'_R(0)$ and $0<\tilde y< \frac{ R_0}{4}$ and consider a corresponding point
$$ \check{Q}\,:=\,(\check{x}',0)$$
in  such a way  that
$$
\check{x}'\in B'(0,R)\qquad\text{and}\qquad
\tilde P \in \partial B_{ R_0}(\check{Q})\,.
$$
Recalling the choice $2R_0<R<L$,
it is easy to check  that such a point exists (and in general is not unique), see Figure \ref{figuuuuu} below.
\begin{figure}[h]
\begin{center}
\includegraphics[height=.55\textwidth,width=0.70\textwidth]{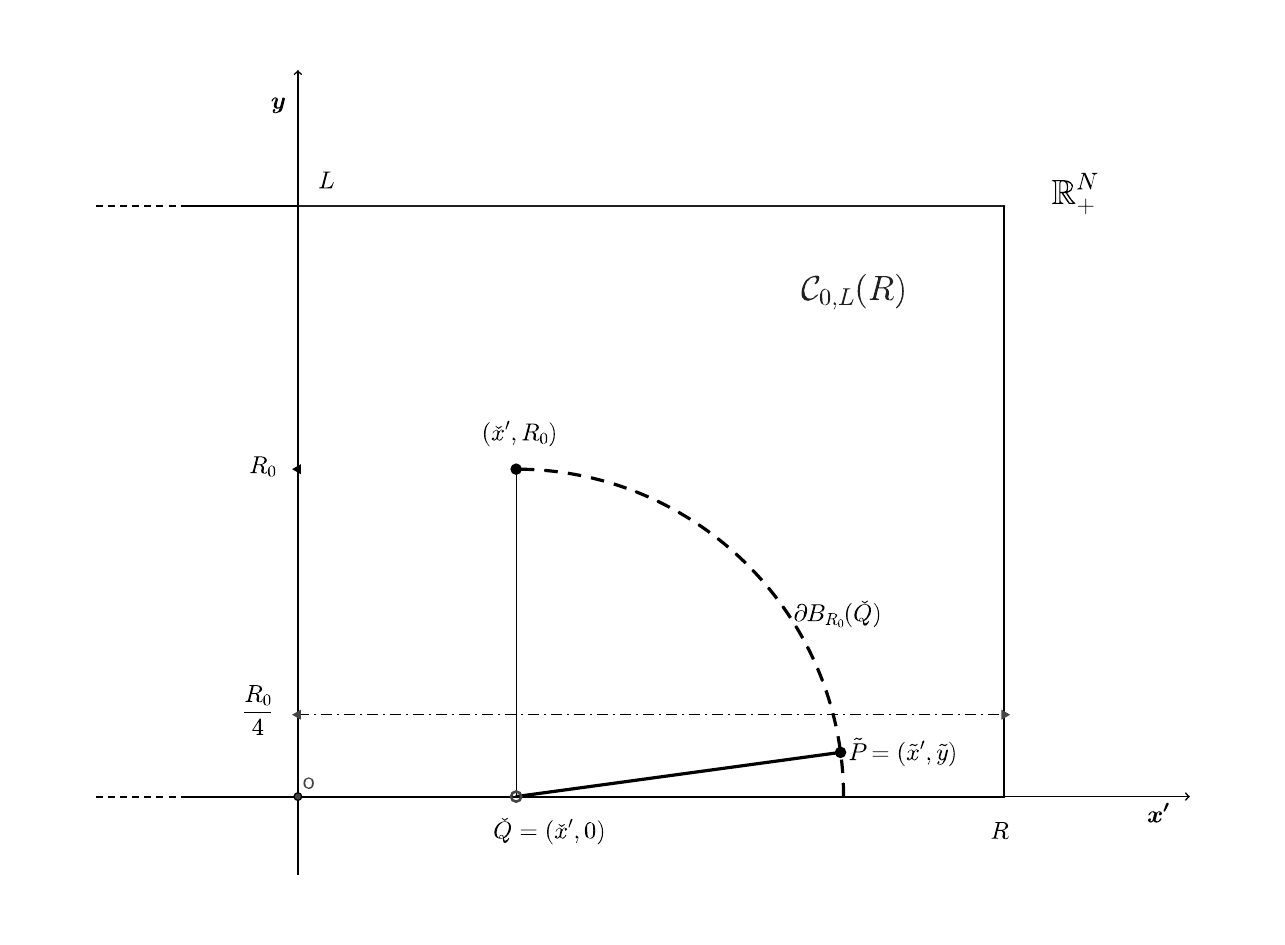}
\caption{\label{figuuuuu}}
\end{center}
\end{figure}
By \cite{bidaut} (see Theorem \ref{thm:BoundaryHarnack}) and recalling \eqref{eq:lacarrozzellaxxx}, we infer  that
\begin{equation}\nonumber
\frac{w_n(\tilde P)}{\tilde y}\leq C\,\frac{w_n(\check{x}',  R_0)}{ R_0},
\end{equation}
and, recalling   also that $w_n(x,0)=0$, we deduce that
$$\|w_n\|_{L^{\infty}(\mathcal C_{(0,L)}(R)\cap \{y\leq \frac{ R_0}{4}\}))}\leq \frac{C }{4}\cdot C^{i}_H(L,R,R_0),$$
that is \eqref{eq:boundaryHarnack} holds, with $C^{b}_H(L,R,R_0)=C \cdot C^{i}_H(L,R,R_0)$.
Finally using \eqref{eq:classicHarnack} and \eqref{eq:boundaryHarnack} it follows that
$$\|w_n\|_{L^{\infty}(\mathcal C_{(0,L)}(R))}\leq C(L,R,R_0).$$
Now  consider $u$, (and consequently $u(x'+x_n', y)$ in \eqref{eq:ashgghagha})  defined on the entire space $\mathbb{R}^N$ by odd reflection. That is
 $$u(x',y)=-u(x',-y) \quad\text{ in } \quad \{y<0\},$$
and consequently
$$f(t)=-f(-t)\quad\text{if } \quad \{t<0\}.$$
In this case we will refer to the cylinder
$$\mathcal C_{(-L,L)}(R):= B'_R(0)\times (-L,L).$$
By standard regularity theory, see e.g Theorem 1 in \cite{T}, since  $\|w_n\|_{L^{\infty}(\mathcal C_{(-L,L)}(R))}\leq C(L,R,R_0)$, we have that
 $$\|w_n\|_{C^{1,\alpha}_{loc}(\mathcal C_{(-L,L)}(R))}\leqslant C(L,R,R_0)$$
for some $0<\alpha<1$.
This allows to use Ascoli-Arzel\`{a} theorem and get that
\begin{equation}\nonumber
w_n\overset{C^{1,\alpha'}_{loc}(\mathcal C_{(-L,L)}(R))}{\longrightarrow}w_{0}
\end{equation}
up to subsequences, for $\alpha '<\alpha$. Furthermore, thanks to \eqref{eq:lacarrozzellaxxx}, we infer that
\begin{equation}\label{weklystar}
c_n(\cdot)\rightarrow\,c_0(\cdot)
\end{equation}
weakly star in $L^{\infty}(\mathcal C_{(-L,L)}(R))$  up to subsequences. This and the fact that $w_0\in C^{1,\alpha'}(\mathcal C_{(-L,L)}(R))$ allows to deduce easily that
\begin{equation}\nonumber
\begin{cases}
-\Delta_p w_0= c_0(x)\,w_0^{p-1} & \text{ in }\mathcal C_{(0,L)}(R)\\
w_0(x',y) \geqslant 0 & \text{ in } \mathcal C_{(0,L)}(R)\\
w_0(x',0)=0&  \text{ on }\partial\mathcal C_{(0,L)}(R)\cap \partial  \mathbb{R}^N_+.
\end{cases}
\end{equation}
By the strong maximum principle, and recalling that
  $w_n(0,1)=1$ for all $n\in \mathbb{N}$,  we deduce that $w_0> 0$ in $\mathcal C_{(0,L)}(R)$ and,  by Hopf's Lemma,  we infer that
$$ \frac{\partial w_0}{\partial y}(0,0)> 0\,.$$
We conclude the proof noticing that
 a contradiction occurs since by
\eqref{eq:salisulcavallo} we should have that
$ \frac{\partial w_0}{\partial y}(0,0)\leq 0$.
\end{proof}
\begin{corollary}\label{cor:monotonia}
Under the hypotheses of Theorem \ref{lem:saliacavallo},  there exits $\lambda>0$ such that, for all $0<\displaystyle \theta\leq \frac{\lambda}{2}$, it holds  that
\begin{equation}\nonumber
u\leq u_{\theta}  \quad \text{in}\,\, \Sigma_\theta.
\end{equation}
\end{corollary}
\begin{proof}
Given $\lambda$ from Theorem \ref{lem:saliacavallo},  using  \eqref{eq:garyyy}, it is sufficient to recall the definition of $u_\theta$ in~\eqref{eq:reflected}.
\end{proof}
Now we prove a technical result, we may refer to it as a weak comparison principle in narrow domains, that we are going to use in  the sequel to prove our main result. We define the projection  $\mathcal P$ as
\begin{eqnarray}\nonumber
\mathcal P:\quad  \mathbb{R}^N&\longrightarrow& \mathbb{R}^{N-1}\\\nonumber(x',y)&\longrightarrow& x'\,.
\end{eqnarray}

\noindent In the proof of the next proposition,  we will use the following  inequalities: \\

\noindent $\forall \eta, \eta' \in  \mathbb{R}^{N}$ with $|\eta|+|\eta'|>0$ there exists positive constants $\dot C, \check C$ depending on $p$ such that
\begin{eqnarray}\label{eq:inequalities}
[|\eta|^{p-2}\eta-|\eta'|^{p-2}\eta'][\eta- \eta'] \geq \dot C (|\eta|+|\eta'|)^{p-2}|\eta-\eta'|^2, \\ \nonumber\\\nonumber
||\eta|^{p-2}\eta-|\eta'|^{p-2}\eta '|\leq \check C (|\eta|+|\eta'|)^{p-2}|\eta-\eta '|.
\end{eqnarray}

\begin{proposition}\label{lem:khjadsjhdfhjdf}
Let $p>2$ and let $u\in C^{1,\alpha}_{loc} (\overline{\mathbb{R}^N_+})$ be a positive weak  solution to \eqref{E:P} with $|\nabla u|\in L^{\infty}(\mathbb{R}^N_+)$.  For
  $0\leq\alpha<\beta\leq \lambda$, let  $\Sigma_{(\alpha,\beta)}$ be the strip defined in \eqref{eq:strippppppp} and
assume  that
\begin{equation}\label{eq:stomp}
u\leq u_\lambda \qquad\text{on}\quad \partial \Sigma_{(\alpha,\beta)}\,.
\end{equation}
Assume furthermore that, setting
\begin{equation}\nonumber
\mathcal I_{(\lambda)}^+=\Big\{(x',\lambda)\,:\, x'\in \mathcal P\big(Supp\,(u-u_{\lambda})^+\big) \Big\},
\end{equation}
it holds that
\begin{equation}\label{ipag}
u(x)\geq \gamma>0 \quad\text{on}\quad \mathcal I_{(\lambda)}^+.
\end{equation}
Then,
for $\Lambda>0$ fixed  such that
\[
\Lambda\geq 2\lambda+1\,,
\]
it follows that  there exists $h_0=h_0(f, p, \gamma, N, \|\nabla u\|_{L^\infty(\Sigma_\Lambda)})$ such that if $\beta -\alpha \leq h_0$ then we have
$$ u\leq u_\lambda\quad \text{in}\quad \Sigma_{(\alpha,\beta)}.$$
\end{proposition}
\begin{proof}
Recalling that $u_{\lambda}(x',y)=u(x',2\lambda -y)$,  we remark that  $(u-u_\lambda)^+ \in L^{\infty}(\Sigma_{(\alpha,\beta)})$ since we assumed
$|\nabla u|$ is bounded. Let us now define
\begin{equation}\nonumber
\Psi=(u-u_\lambda)^+\varphi_R^2,\end{equation} where $\varphi_R(x',y)=\varphi_R(x') \in C^{\infty}_c (\mathbb{R}^{N-1}) $, $\varphi_R \geq 0$ such that
\begin{equation}\label{Eq:Cut-off1}
\begin{cases}
\varphi_R \equiv 1, & \text{ in } B^{'}(0,R) \subset \mathbb{R}^{N-1},\\
\varphi_R \equiv 0, & \text{ in } \mathbb{R}^{N-1} \setminus B^{'}(0,2R),\\
|\nabla \varphi_R | \leq \frac CR, & \text{ in } B^{'}(0, 2R) \setminus B^{'}(0,R) \subset  \mathbb{R}^{N-1},
\end{cases}
\end{equation}
where $B^{'}(0,R)$ denotes the ball in $\mathbb{R}^{N-1}$ with center $0$ and radius $R>0$. From now on, for the sake of simplicity, we set $\varphi_R(x',y):=\varphi(x',y)$.
By \eqref{Eq:Cut-off1} and  by the fact that  $u\leq u_\lambda  \text{ on } \partial \Sigma_{(\lambda, \beta)}$ (see \eqref{eq:stomp}),   it follows that
$$\Psi \in W_0^{1,p}(\mathcal C_{(\alpha,\beta)}(2R)).$$
Since $u$ is a solution to problem \eqref{E:P}, then it follows that $u,u_\lambda$ are solutions to
\begin{equation}\label{Eq:WCP}
\begin{cases}
-\Delta_p u= f(u) & \text{ in }\Sigma_{(\alpha,\beta)},\\
- \Delta_p u_\lambda=f(u_\lambda) & \text{ in }\Sigma_{(\alpha,\beta)},\\
\qquad u\leq u_\lambda & \text{ on } \partial \Sigma_{(\alpha,\beta)}.
\end{cases}
\end{equation}
Then using $\Psi$ as test function in both equations of problem \eqref{Eq:WCP} and substracting we get
\begin{eqnarray}\label{eq:cn1}\\\nonumber
&&\int_{\mathcal{C}(2R)}\big(|\nabla u|^{p-2}\nabla u-|\nabla u_\lambda|^{p-2}\nabla u_\lambda,\nabla(u-u_\lambda)^+\big)\varphi^2\\\nonumber&+&\int_{\mathcal{C}(2R)}\big(|\nabla u|^{p-2}\nabla u-|\nabla u_\lambda|^{p-2}\nabla u_\lambda,\nabla \varphi^2\big)(u-u_\lambda)^+\\\nonumber
&= &\int_{\mathcal{C}(2R)}\big(f(u)-f(u_\lambda)\big)(u-u_\lambda)^+\varphi^2,
\end{eqnarray}
where $\mathcal{C}(\cdot)$ denotes the cylinder defined in  \eqref{eq:cylllllllinder}.
 By \eqref{eq:inequalities} and the fact that $p \geq 2$, from \eqref{eq:cn1} we deduce that
\begin{eqnarray}\label{eq:cn2}
&&\dot  C\int_{\mathcal{C}(2R)}(|\nabla u|+|\nabla u_\lambda|)^{p-2}|\nabla(u-u_\lambda)^+|^2\varphi^2 \\\nonumber &\leq&\int_{\mathcal{C}(2R)}\big(|\nabla u|^{p-2}\nabla u-|\nabla u_\lambda|^{p-2}\nabla u_\lambda,\nabla(u-u_\lambda)^+\big)\varphi^2\\\nonumber
&=& - \int_{\mathcal{C}(2R)}\big(|\nabla u|^{p-2}\nabla u-|\nabla u_\lambda|^{p-2}\nabla u_\lambda,\nabla \varphi^2\big)(u-u_\lambda)^+\\\nonumber &+&\int_{\mathcal{C}(2R)}\big(f(u)-f(u_\lambda)\big)(u-u_\lambda)^+\varphi^2\\\nonumber &\leq&
\int_{\mathcal{C}(2R)}\left |\big(|\nabla u|^{p-2}\nabla u-|\nabla u_\lambda|^{p-2}\nabla u_\lambda,\nabla \varphi^2\big)\right|(u-u_\lambda)^+ \\\nonumber &+&\int_{\mathcal{C}(2R)}\big(f(u)-f(u_\lambda)\big)(u-u_\lambda)^+\varphi^2\\\nonumber
&\leq&\check C\int_{\mathcal{C}(2R)} (|\nabla u|+|\nabla u_\lambda|)^{p-2}|\nabla(u-u_\lambda)^+| |\nabla \varphi^2| (u-u_\lambda)^+ \\\nonumber &+&\int_{\mathcal{C}(2R)}\big(f(u)-f(u_\lambda)\big)(u-u_\lambda)^+\varphi^2,\end{eqnarray}
where in the last line we used Schwarz inequality and the second of \eqref{eq:inequalities}. Setting
\begin{equation}\label{eq:I1}{ I_1:=\check C\int_{\mathcal{C}(2R)} (|\nabla u|+|\nabla u_\lambda|)^{p-2}|\nabla(u-u_\lambda)^+| |\nabla \varphi^2| (u-u_\lambda)^+}
\end{equation}
and
\begin{equation}\label{eq:I2'}
I_2:=\int_{\mathcal{C}(2R)}\big(f(u)-f(u_\lambda)\big)(u-u_\lambda)^+\varphi^2,
\end{equation}
 \eqref{eq:cn2} becomes
\begin{equation}\label{eq:cn33} \dot  C\int_{\mathcal{C}(2R)}(|\nabla u|+|\nabla u_\lambda|)^{p-2}|\nabla(u-u_\lambda)^+|^2\varphi^2\leq I_1+I_2.
\end{equation}
In order to estimate the terms $I_1$ and $I_2$ in \eqref{eq:cn33}
we will exploited the weighted Poincar\'e type inequality \eqref{eq:poincareineq} (see \cite{DS1}) and a covering argument that goes back to \cite{FMS3}.
Let us consider  the  hypercubes  $Q_i$  of $\mathbb{R}^N$ defined by
$$Q_i=Q_i'\times [\alpha,\beta],$$
where $Q_i'\subset \mathbb{R}^{N-1}$ are hypercubes of $\mathbb{R}^{N-1}$, with edge $\beta-\alpha$ and   such that
$$\bigcup_i Q_i'=\mathbb{R}^{N-1}.$$
Moreover we assume that   $Q_i\cap Q_j=\emptyset$ for $i\neq j$ and
\begin{equation}\label{eq:Qunion}
\bigcup_{i=1}^{\overline{N}}\overline{Q_i}\supset \mathcal{C}(2R).
\end{equation}
It follows as well,  that each set  $Q_i$ has diameter
\begin{equation}\label{eq:diameterQ}
\text{diam}(Q_i)=d_Q=\sqrt{N}(\beta-\alpha), \qquad  i=1,\cdots,\overline{N}.
\end{equation}
 The  covering  in \eqref{eq:Qunion} will allow us  to use in each  $Q_i$ the weighted Poincar\'e type inequality  and to  take advantage of the constant  $C_p$ in   Theorem \ref{thm:Poincar},  that turns to be not depending on the index $i$ of \eqref{eq:Qunion}. Later we will recollect the estimates.\\

   \noindent Let us define
\begin{equation}\label{eq:w}
w(x):=
\begin{cases}
\Big(u-u_\lambda\Big)^+(x',y) & \text{if } (x',y)\in \overline{Q}_i; \\
-\Big(u-u_\lambda\Big)^+(x',2\beta-y) & \text{if } (x',y)\in  \overline{Q}_i^r,
\end{cases}
\end{equation}
where $(x',y)\in \overline{Q}_i^r$ iff $(x',2\beta-y)\in \overline{Q}_i$.
We claim that
\begin{equation}\label{eq:wkfjbnkdfhsdkb}
\int_{Q_i}\,w^2\,\leq\,C_p(Q_i)\,\int_{Q_i} (|\nabla u|+|\nabla u_\lambda|)^{p-2} |\nabla w|^2
\end{equation}
where $C_p(Q_i)$ is given by Theorem \ref{thm:Poincar} and has the property that it goes to zero if the diameter of $Q_i$
goes to zero. Actually, since $p\geq2$, we will deduce  \eqref{eq:wkfjbnkdfhsdkb} by
\begin{equation}\label{eq:wkfjbnkdfhsbmgkdkb}
\int_{Q_i}\,w^2\,\leq\,C_p(Q_i)\,\int_{Q_i} |\nabla u_\lambda|^{p-2} |\nabla w|^2\,.
\end{equation}
The fact that Theorem \ref{thm:Poincar} can be applied to deduce \eqref{eq:wkfjbnkdfhsbmgkdkb} is somehow technical and we describe the procedure here below.\\

\noindent We have ${\displaystyle{\int_{Q_i \cup Q_i^r}w(x)dx=0}}$ and therefore, see \cite[Lemma 7.14, Lemma 7.16]{GT}, it follows that
$$w(x)=\hat C\int_{{Q_i \cup Q_i^r}}\frac{(x_i-z_i)D_iw(z)}{|x-z|^N}dz\quad  \,\,\text{a.e. } x\in {Q_i \cup Q_i^r},$$
where $\hat C= \hat C(d_Q,N)$, is a positive constant. Then for almost every $ x\in {Q_i}$ we have

\begin{eqnarray}\nonumber
|w(x)|&\leq& \hat C\int_{{Q_i \cup Q_i^r}}\frac{|\nabla w(z)|}{|x-z|^{N-1}}dz\\\nonumber
&=&\hat C\int_{{Q_i}}\frac{|\nabla w(z)|}{|x-z|^{N-1}}dz+\hat C\int_{{Q_i^r}}\frac{|\nabla w(z)|}{|x-z|^{N-1}}dz\\\nonumber
&\leq& 2\hat C\int_{{Q_i}}\frac{|\nabla w(z)|}{|x-z|^{N-1}}dz\,,
\end{eqnarray}
where in the last line we used the following  standard changing of variables
$$(z^t)'=z'\quad\text{and}\quad z_N^t=2\beta - z_N,$$
the fact that for $x\in Q_i$, it holds that $(|x-z|)\Big |_{z\in Q_i}\leq (|x-z^t|)\Big |_{z\in Q_i}$ and that, by \eqref{eq:w} it holds that
$|\nabla w(z)|=|\nabla w(z^t)|$. \\

 \

\noindent  Hence  \eqref{eq:hip_real} holds and, in order to prove \eqref{eq:wkfjbnkdfhsbmgkdkb}, we need to show that
\eqref{eq:weight} holds with
\[
\rho\,:=\,|\nabla u_\lambda|^{p-2}\,.
\]
Note now that, if $w$ vanishes identically in $Q_i$, then there is nothing to prove. If not it is easy to see that by our assumptions (see \eqref{ipag}) and by the classical Harnack inequality, it follows  that there exists $\bar \gamma >0$ such that
\begin{equation}\label{ipaggtt}
u\geq\bar \gamma>0\qquad\text{in}\quad \tilde Q_i'\times [\lambda/2\,,\,4\lambda]
\end{equation}
where
\[
\tilde Q_i'\,:=\,\{x\in\mathbb{R}^{N-1}\,\,:\,\,dist(x,Q_i')<1\}\,.
\]
 Let us consider  $Q_i^{\mathcal R_{\lambda}}$ obtained by the reflection  of $Q_i$ with respect to the hyperplane $T_\lambda=\{(x',y)\in \mathbb{R}^N\,:\ y=\lambda\}$. Since $Q_i^{\mathcal R_{\lambda}}$ is bounded away from the boundary $\mathbb{R}^N$, namely
$$\text{dist}\,(Q_i^{\mathcal R_{\lambda}} \,,\,\{y=0\})\geq \lambda>0,$$
thanks to \eqref{ipaggtt}
then   Proposition \ref{pro:SobConstant} apply with 
\begin{equation}\nonumber
\beta_1=\underset{t\in [\bar\gamma,\|u\|_{L^\infty(\Sigma_\Lambda)}]}{\min}\,f(t)\qquad \text{and}\qquad \beta_2=\lambda
\end{equation}
and we obtain that
\begin{eqnarray}\nonumber
\int_{Q_i^{\mathcal R_{\lambda}}}\frac{1}{|\nabla u|^{p-2}} \frac{1}{|x-y|^{\gamma}}\,dy\leq C^*_1(\beta_1,\beta_2)\qquad \text{for any}\quad x\in Q_i^{\mathcal R_{\lambda}}.
\end{eqnarray}	
By symmetry we deduce therefore that
\begin{eqnarray}\nonumber
\int_{Q_i}\frac{1}{|\nabla u_\lambda|^{p-2}} \frac{1}{|x-y|^{\gamma}}\,dy\leq C^*_1(\beta_1,\beta_2)\qquad \text{for any}\quad x\in Q_i,
\end{eqnarray}
so that we can exploit Theorem \ref{thm:Poincar} to deduce \eqref{eq:wkfjbnkdfhsbmgkdkb} and consequently \eqref{eq:wkfjbnkdfhsdkb}.
\

\noindent Let us now estimate the R.H.S. of  \eqref{eq:cn33}.  Recalling \eqref{eq:I1} we get
\begin{eqnarray}\nonumber
\\\nonumber
I_1&=&2\check C\int_{\mathcal{C}(2R)} (|\nabla u|+|\nabla u_\lambda|)^{p-2}|\nabla(u-u_\lambda)^+| \varphi|\nabla \varphi| (u-u_\lambda)^+ \\\nonumber
&=&2\check C\int_{\mathcal{C}(2R)} (|\nabla u|+|\nabla u_\lambda|)^{\frac{p-2}{2}}|\nabla(u-u_\lambda)^+| \,\varphi  (|\nabla u|+|\nabla u_\lambda|)^{\frac{p-2}{2}}|\nabla \varphi|  (u-u_\lambda)^+\\\nonumber &\leq&\delta' \check C\int_{\mathcal{C}(2R)}(|\nabla u|+|\nabla u_\lambda|)^{p-2}|\nabla(u-u_\lambda)^+|^2
\varphi^2
\\\nonumber
&+&\frac{\check C}{\delta'}\int_{\mathcal{C}(2R)}(|\nabla u|+|\nabla u_\lambda|)^{p-2}
|\nabla\varphi|^2[(u-u_\lambda)^+]^2,
\end{eqnarray} where in the last inequality  we used weighted Young inequality, with $\delta'$ to be chosen later.
 Hence
 \begin{equation}\label{eq:supI_1}
 I_1\leq I_1^a+I_1^b,
 \end{equation} where
 \begin{eqnarray}\label{eq:supI_1^a}
 I_1^a&:=&\delta' \check C\int_{\mathcal{C}(2R)}(|\nabla u|+|\nabla u_\lambda|)^{p-2}|\nabla(u-u_\lambda)^+|^2
\varphi^2,\\\nonumber
I_1^b&:=&\frac{\check C}{\delta'}\int_{\mathcal{C}(2R)}(|\nabla u|+|\nabla u_\lambda|)^{p-2}
|\nabla\varphi|^2[(u-u_\lambda)^+]^{2}.
 \end{eqnarray}
Using the covering in \eqref{eq:Qunion},  the properties of the cut-off function in \eqref{Eq:Cut-off1}  and   the fact  that $|\nabla u|$ and  $|\nabla u_\lambda| $ are bounded, by \eqref{eq:wkfjbnkdfhsdkb} we deduce that
\begin{eqnarray}\label{eq:cn4}
I_1^b&\leq&\sum_{i=1}^{\overline{N}}\frac{C}{\delta' R^2}\int_{\mathcal{C}(2R) \cap Q_i }[(u-u_\lambda)^+]^2\\\nonumber&\leq& \max_i C_P(Q_i)\sum_{i=1}^{\overline{N}}\frac{C}{\delta' R^2}\int_{\mathcal{C}(2R) \cap Q_i }(|\nabla u|+|\nabla u_\lambda |)^{p-2}|\nabla(u-u_\lambda)^+|^2\\\nonumber&\leq&  C^*_P\frac{C}{\delta' R^2}\int_{\mathcal{C}(2R) }(|\nabla u|+|\nabla u_\lambda |)^{p-2}|\nabla(u-u_\lambda)^+|^2
\end{eqnarray}
 where $C^*_P=\max_i C_P(Q_i)$ and $C=C(p, \|\nabla u\|_{L^\infty(\Sigma_\Lambda)})$.
\\

\noindent Now we estimate the term $I_2$ in \eqref{eq:cn33}.
Being $f$ locally Lipschitz continuous
form  \eqref{eq:I2'}, arguing as in \eqref{eq:cn4}, we get  that
\begin{eqnarray}\nonumber
I_2&\leq&\int_{\mathcal{C}(2R)}\frac{f(u)-f(u_\lambda)}{u-u_\lambda}[(u-u_\lambda)^+]^2\\\nonumber
&\leq&  C^*_P\cdot C\int_{\mathcal{C}(2R) }(|\nabla u|+|\nabla u_\lambda |)^{p-2}|\nabla(u-u_\lambda)^+|^2,
\end{eqnarray}
where $C^*_P$ is as in \eqref{eq:cn4} and $C=C(f,\lambda,  \|\nabla u\|_{L^\infty(\Sigma_\Lambda)})$. Actually the constant $C$ will depend on the Lipschitz constant of $f$ in the interval $\big[0,\max\{\|u\|_{L^{\infty}(\Sigma_{\Lambda})},\|u_\lambda\|_{L^{\infty}(\Sigma_{\Lambda})}\}\big].$
By~\eqref{eq:cn33}, \eqref{eq:supI_1},  \eqref{eq:supI_1^a} and \eqref{eq:cn4}, up to redefining the constants, we obtain
\begin{eqnarray}\label{eq:L1}
&&C\int_{\mathcal{C}(2R)}(|\nabla u|+|\nabla u_\lambda|)^{p-2}|\nabla(u-u_\lambda)^+|^2\varphi^2\\\nonumber
&\leq&\delta' \int_{\mathcal{C}(2R)}(|\nabla u|+|\nabla u_\lambda|)^{p-2}|\nabla(u-v)^+|^2
\\\nonumber
&+&\frac{C^*_P}{R}\int_{\mathcal{C}(2R)}(|\nabla u|+|\nabla u_\lambda|)^{p-2 }|\nabla(u-u_\lambda)^+|^2 \\\nonumber&+& C^*_P\int_{\mathcal{C}(2R)}(|\nabla u|+|\nabla u_\lambda|)^{p-2}|\nabla(u-u_\lambda)^+|^2.
\end{eqnarray}
Let us choose $\delta'$ small  in \eqref{eq:L1}  such that
$C-\delta' >C/2$ and fix $R>1$. Then we obtain
\begin{eqnarray}\label{eq:L111111}
&&\int_{\mathcal{C}(2R)}(|\nabla u|+|\nabla u_\lambda|)^{p-2}|\nabla(u-u_\lambda)^+|^2\varphi^2\\\nonumber
&\leq&4\frac{C^*_P}{C}\int_{\mathcal{C}(2R)}(|\nabla u|+|\nabla u_\lambda|)^{p-2 }|\nabla(u-u_\lambda)^+|^2 .
\end{eqnarray}
To conclude we  set now
\begin{equation}\label{eq:L1111111}
 \mathcal {L}(R)\,:=\,\int_{\mathcal{C}(R)}(|\nabla u|+|\nabla u_\lambda|)^{p-2}|\nabla(u-u_\lambda)^+|^2.
\end{equation}
We can fix $h_0=h_0(f,p,\gamma,\lambda, N,  \|\nabla u\|_{L^\infty(\Sigma_\Lambda)})$ positive, such that if
$$\beta -\alpha \leq h_0,$$ (recall that $C^*_P\rightarrow 0$ in this case since    $\text{diam}(Q_i)\rightarrow 0$, see \eqref{eq:diameterQ}) then
$$\displaystyle \theta:=4\frac{C^*_P}{C}< 2^{-N}.$$
Then, by \eqref{eq:L111111} and \eqref{eq:L1111111}, we have
$$
\begin{cases}
\mathcal{L}(R)\leq \theta \mathcal{L}(2R)& \forall R>1,\\
\mathcal{L}(R)\leq CR^{N} & \forall R >1.
\end{cases}
$$
>From Lemma \ref{Le:L(R)} with $\nu=N$ and $\teta<2^{-N}$, we get $$\mathcal{L}(R)\equiv0$$ and consequently that $(u-u_\lambda)^+\equiv 0$.
\end{proof}
The proof of our main result will follow by the moving plane procedure that will be strongly based on Proposition \ref{lem:khjadsjhdfhjdf}. As it will be clear later,
it will be needed to substitute $\lambda$ by $\lambda+\varepsilon$ in order to proceed further from the maximal position. To do this we need to be very accurate in the estimate of the constants involved, namely we need to control role of $h_0$ in Proposition \ref{lem:khjadsjhdfhjdf}. This is the reason for which  we introduced the larger strip $\Sigma_\Lambda$ that allows to control the needed bound on $|\nabla u|$. But still we need to control the dependence  of $h_0$ on $\gamma$ (see \eqref{ipag}). This can be resumed saying that we need a uniform (with respect to $\varepsilon$) control on the infimum of $u$ far from the boundary, and in the set where $u$ is greater than $u_\lambda$. This motivates the following
\begin{lemma}\label{cacsasddcarlll}
Let $\lambda >0$ and let $u$ be  a solution to \eqref{E:P}, with $|\nabla u|\in L^{\infty}(\mathbb{R}^N_+)$ and $u_\lambda$ defined  as in \eqref{eq:reflected}. Assume here that ($h_f$) is fulfilled with $f_0=0$ and
define
\begin{equation}\nonumber
\mathcal I_{(\lambda,\varepsilon)}^+=\Big\{(x',\lambda)\,:\, x'\in \mathcal P\big(Supp\,(u-u_{\lambda+\varepsilon})^+\big) \Big\}.
\end{equation}
Then there exist  $\varepsilon_0>0$ and $\gamma >0$   such that
\begin{equation}\nonumber
u(x)\geq \gamma \quad\text{on}\quad \mathcal I_{(\lambda,\varepsilon)}^+,
\end{equation}
for all $0\leq\varepsilon\leq \varepsilon_0$.
\end{lemma}
\begin{proof}If the result is not true, then
by contradiction given $\varepsilon_0 >0$ and $\gamma >0$,  we found $0\leq\varepsilon\leq \varepsilon_0$  and  a point $Q_\varepsilon=(x'_\varepsilon,\lambda)$ with $Q_\varepsilon\in \mathcal I_{(\lambda,\varepsilon)}^+$ such that
\begin{equation}\nonumber
u(x'_\varepsilon,\lambda)\leq \gamma.
\end{equation}
It is convenient to consider  $\varepsilon_0=\gamma=1/n$ and the corresponding  $\varepsilon=\varepsilon_n\leq\varepsilon_0$ defined by contradiction as above,
that obviously approaches zero as $n$ tends to infinity. Also we use the notation $Q_{\varepsilon_n}\in \mathcal I_{(\lambda,\varepsilon_n)}^+$.
On a corresponding sequence $P_n=(x_n',y_n)$ we have that
\begin{equation}\label{eq:hsakhdhgjdvgc}
u(x_n',y_n)\geq u_{ \lambda+\varepsilon_n}(x_n',y_n)\quad {\mbox{ with }} (x_n',y_n)\in\Sigma_{\lambda+\varepsilon_n},
\end{equation}
where the existence of the sequence $(x'_n,y_n)$ follows by the fact that $Q_{\varepsilon_n}\in \mathcal I_{(\lambda,\varepsilon_n)}^+$
and
(up to subsequences) $$y_n\rightarrow y_0\in [0,\lambda].$$
Moreover
\begin{equation}\nonumber
\lim_{n\rightarrow+\infty} u(x'_n,\lambda)\rightarrow 0.
\end{equation}
Let us  set
\begin{equation}\label{eq:ashgghaghabbb}
w_n(x',y)=\frac{u(x'+x_n', y)}{\alpha_n}\end{equation}
and
\begin{equation}\nonumber
\alpha_n:= u(x_n',\lambda),
\end{equation}
with $\displaystyle \lim_{n\rightarrow +\infty}\alpha_n=0$.
We remark that $w_n(0,\lambda)=1$. Then we have
\begin{eqnarray}\label{sdjkfgsfgkbbb}
-\Delta_pw_n(x)
=c_n(x)w_n^{p-1}(x),
\end{eqnarray}
for
\begin{equation}\label{eq:cccccc(x)}
c_n(x)=\frac{f\big(u(x'+x_n', y)\big)}{u^{p-1}(x'+x_n', y)}\,.
\end{equation}
Since for any $L>0$ we have that $u\in L^{\infty}(\Sigma_{(L)})$ (by the Dirichlet condition and because $|\nabla u|$ is bounded in $\mathbb{R}^N_+$),  by $(h_f)$ we obtain that
\begin{equation}\label{eq:lacarrozzella}
\|c_n(x)\| _{L^{\infty}(\Sigma_{L})}\leq C(L).
\end{equation}
For $L>\lambda$ we consider the cylinder
$\mathcal C_{(0,L)}(R)$ and,
arguing as in the proof of Theorem \ref{lem:saliacavallo} (see the first claim there), we deduce that
\begin{equation}\nonumber
\|w_n\|_{L^{\infty}(\mathcal C_{(0,L)}(R))}\leq C(L)\,.
\end{equation}
Now, as in the proof of Theorem \ref{lem:saliacavallo}, we  consider $u$ defined on the entire space $\mathbb{R}^N$ by odd reflection and,
by standard regularity theory (see \cite{Di,T}), we deduce that
 $$\|w_n\|_{C^{1,\alpha}_{loc}(\mathcal C_{(-L,L)}(R))}\leqslant C(L)$$
for some $0<\alpha<1$. This allows to use Ascoli-Arzel\`{a} theorem and get
\begin{equation}\nonumber
w_n\overset{C^{1,\alpha'}_{loc}(\mathcal C_{(-L,L)}(R))}{\longrightarrow}w_{L,R}
\end{equation}
up to subsequences, for $\alpha '<\alpha$.
Replacing $L$ by $L+n$ ($n\in\mathbb{N}$), and $R$ by $R+n$ we can repeat the argument above and then perform a standard diagonal process
to define
 $w$  in the entire space $\mathbb{R}^N$ in such a way that $w$ is locally the limit of subsequences of $w_n$. It turns out that, by construction, setting
  $$w_+(x) =w(x)\cdot \chi_{\overline{\mathbb{R}^N_+}} $$
  we have that
\begin{equation}\nonumber
\begin{cases}
-\Delta_p w_+=0, & \text{ in }\mathbb{R}^N_+\\
w_+(x',y) \geqslant 0, & \text{ in } \mathbb{R}^N_+\\
w_+(x',0)=0,&  \text{ on }\partial\mathbb{R}^N_+\,.
\end{cases}
\end{equation}
This is a simple computation
where in \eqref{sdjkfgsfgkbbb} we need to use the fact that $c_n(x)\rightarrow 0$ as $n\rightarrow +\infty$ uniformly on compact sets. This follows
in fact considering that $w_n$ is uniformly bounded on compact sets and then, by \eqref{eq:ashgghaghabbb} it follows that $u(x+x'_n,y)\rightarrow 0$ as $n\rightarrow +\infty$.
By \eqref{eq:cccccc(x)} and recalling that
$$\lim_{t\rightarrow 0}\frac{f(t)}{t^{p-1}}=0,$$
finally it follows that $c_n(x)\rightarrow 0$ on compact sets.\\
\noindent By the strong maximum principle,  we have now  that $w_+>0$, in view of the fact that (by uniform convergence of $w_n$)
$w_+(0,\lambda)=1$.
By \cite[Theorem 3.1]{KSZ}, it follows  that $w_+$ must be affine linear, i.e $w_+(x',y)=ky$, for some $k>0$ by the Dirichlet condition.
If $y_0\in [0,\lambda)$, by \eqref{eq:hsakhdhgjdvgc}  and by the uniform convergence of $w_n\rightarrow w_+$, we would have
\begin{equation}\nonumber
w_+(0,y_0)\geq {(w_+)}_{\lambda}(0,y_0).
\end{equation}
This is a contradiction since $w_+(x',y)=ky$ for some $k>0$. \\

\noindent Therefore let us assume that  $y_n\rightarrow \lambda$ and note that,  by the mean value theorem, at some point $\xi_n$ lying on the segment from
 $(0,y_n)$ to $(0, 2(\lambda+\varepsilon_n)-y_n)$, it should hold that
\begin{eqnarray}\nonumber
\frac{\partial w_n}{\partial y}(0, \xi_n)\leq0\,.
\end{eqnarray}
Since  $w_n\rightarrow w_+$  in $C^{1,\alpha}_{loc}(\overline{\mathbb{R}^N_+})$ we would have that
$$ \frac{\partial w_+}{\partial y}(0,\lambda)\leq 0\,.$$
Again this is a contradiction since $w_+(x',y)=ky$, for some $k>0$ and the result is proved.
\end{proof}

\noindent The results proved above allow us to conclude the proof of our main result.\\
\begin{proof}[Proof of Theorem \ref{thm:monotoniadellau}]
We consider here the case  when ($h_f$) is fulfilled with $f_0=0$ since in the simpler case $f_0>0$ the result follows directly by Theorem 3 in \cite{FMS3}.
 Thanks to  Corollary \ref{cor:monotonia}
 we have that the set
 $$\Lambda\equiv\{t>0\,\, :\,\, u\leqslant u_\alpha\, \quad\text{in}\quad \Sigma_\alpha\, \quad\forall \alpha \leqslant t\},$$
 is not empty. To conclude the proof, if  we set $$\bar{\lambda}=\sup\Lambda$$
  (that now is well defined)   we have to show that
\begin{center}
$\bar\lambda=+\infty$.
\end{center}
By contradiction  assume that $\bar \lambda< +\infty$ and set
$$W^+_{\varepsilon}:=(u-u_{\bar \lambda +\varepsilon})^+\chi_{\Sigma_{\bar \lambda+\varepsilon}}.$$
 We point out that given $0<\delta< \bar\lambda /2$, there exists $\varepsilon_0$ such that for all $0<\varepsilon\leq \varepsilon_0$ it follows that
\begin{equation}\nonumber
Supp \, W^+_{\varepsilon}\subset \Sigma_{\delta}\, \cup \, \Sigma_{(\bar \lambda-\delta, \bar \lambda+\varepsilon)}.
\end{equation}
This follows by an analysis of the limiting profile at infinity. We do not add the details since the proof is exactly the one in  \cite[Proposition 4.1]{FMS}.
\
For $\delta$ and $\varepsilon_0$ sufficiently small  Proposition~\ref{lem:khjadsjhdfhjdf} applies in $\Sigma_{\delta}$ and in $ \Sigma_{(\bar \lambda-\delta, \bar \lambda+\varepsilon)}$ with $\lambda=\bar\lambda+\varepsilon$ and $\Lambda=2\bar\lambda+1$. It is crucial here the fact that,
thanks to Lemma \ref{cacsasddcarlll}, the parameter $h_0$ in the statement of Proposition~\ref{lem:khjadsjhdfhjdf}, can be chosen independently of $\varepsilon$ since there $\gamma$ does not depend on $\varepsilon$. Then we conclude that
 $W_{\varepsilon}^+\equiv 0$. This is   a contradiction with the definition of $\bar \lambda$, so that we have proved that $\bar\lambda=\infty$.
This implies the monotonicity of $u$ in the half-space, that is $
 \frac{\partial u}{\partial y}(x)\geqslant 0$ in  $\mathbb{R}^N_+
 $. By Theorem \ref{hthPMFderrr}, since $u$ is not trivial,  it follows
  \[
 \frac{\partial u}{\partial y}(x)>0 \quad \text{in}\quad \mathbb{R}^N_+.
 \]
Finally, to prove that $u \in C^{2,\alpha'}_{loc}(\overline{\mathbb{R}^N_+})$, just note that, from the fact that $ \frac{\partial u}{\partial y}\,>\,0$, we deduce that  the set of critical points $\{\nabla\,u=0\}$ is empty and consequently the equation is no more degenerate. The $C^{2,\alpha'}$ regularity follows therefore by standard regularity results, see  \cite{GT}.

\end{proof}

\begin{proof}[Proof of Theorem \ref{mainthmnonpositive}]
By Theorem 1.7 in \cite{FMS} it follows that $0<u\leq t_0$. Thanks to the behaviour of the nonlinearity near $t_0$ (see \eqref{dnvkjdvkvkvks}), then the strong maximum principle applies and implies that actually $0<u< t_0$ in the half space. Arguing now as in the proof of Theorem 1.3 in \cite{FMRS} it follows that $u$ is strictly bounded away from $t_0$ in $\Sigma_\lambda$ for any $\lambda>0$. Now the monotonicity of the solution follows by our Theorem \ref{thm:monotoniadellau} (in the case $f_0>0$ the result follows also directly by Theorem 3 of \cite{FMS3}).
Note in fact that   the condition ($h_f$) is  satisfied in the range of values that the solutions takes in any strip and this is sufficient in order to run over again the moving plane procedure.
\end{proof}

\bigskip

\end{document}